\documentclass[12pt]{article}
\usepackage{a4}
\usepackage{latexsym}
\usepackage{amsfonts}
\usepackage{amssymb}
\usepackage{amsmath}
\usepackage{amsthm}

\newtheorem{theorem}{Theorem}

\newtheorem{lemma}[theorem]{Lemma}
\newtheorem{corollary}[theorem]{Corollary}

\newtheorem{definition}[theorem]{Definition}

\newtheorem*{question}{Question}
\title{New results for the growth of sets of real numbers}
\date{}
\author{Timothy G. F. Jones\footnote{School of Mathematics, University of Bristol BS8 1TW, United Kingdom, tgf.jones@bristol.ac.uk.}}
\begin{document}

\maketitle

\abstract{
\setlength{\parindent}{0em}
\setlength{\parskip}{1.5ex plus 0.5ex minus 0.5ex}
We use the theory of cross ratios to construct a real-valued function $f$ of only three variables with the property that for any finite set $A$ of reals, the set $f(A)=\left\{f(a,b,c):a,b,c \in A\right\}$ has cardinality at least $C|A|^2/\log |A|$, for an absolute constant $C$. Previously-known functions with this property had all been of four variables. 

We also improve on the state of the art for  functions of four variables by constructing a function $g$ for which $g(A)$ has cardinality at least $C|A|^2$; the previously best-achieved bound was $C|A|^2/\log |A|$. 

Finally, we give an example of a five-variable function $h$ for which $h(A)$ has cardinality at least $C|A|^4/\log |A|$.

Proving these results depends only on the Szemer\'edi-Trotter incidence theorem and an analoguous result for planes due to Edelsbrunner, Guibas and Sharir, each applied in the Erlangen-type framework of Elekes and Sharir. In particular the proofs do not employ the Guth-Katz polynomial partitioning technique or the theory of ruled surfaces.

Although the growth exponents for $f,g$ and $h$ are stronger than those for previously-considered functions, it is not clear that they are necessarily sharp. So we pose a question as to whether the bounds on the cardinalities of $f(A),g(A)$ and $h(A)$ can be further strengthened.  
} 

\setlength{\parindent}{0em}
\setlength{\parskip}{1.5ex plus 0.5ex minus 0.5ex}

\section{Introduction}

Throughout this paper we use $X=\Omega(Y)$, $Y=O(X)$, and $Y\ll X$ all to mean that there is an absolute constant $C$ with $Y \leq CX$.

\subsection{Growth and expanders}

The concept of growth is a major theme in modern arithmetic combinatorics and combinatorial geometry. The motivating example is the sum-product phenomenon, which says that for a set $A \subseteq \mathbb{R}$, at least one of the sumset $A+A=\left\{a+b:a,b \in A\right\}$ and the product set $AA=\left\{ab:a,b \in A\right\}$ will have cardinality at least $\Omega\left(|A|^{1+\delta}\right)$ for some absolute $\delta>0$. The best-known exponent when $A$ is a subset of $\mathbb{R}$ is $\delta=1/3-o(1)$, due to Solymosi \cite{solymosi}.

Another example of growth, and the one considered here, is that of so-called expander functions. An $n$-variable expander is a function $f$ for which the set 

$$f(A)=\left\{f(a_1,\ldots,a_n):a_i \in A\right\}$$

has cardinality at least $\Omega\left(|A|^{1+\delta}\right)$. The study of such functions was initiated by Bourgain \cite{bourgain} in a  finite field setting, but the strongest results are known for real-valued sets and functions. For example in the $n=2$ case Garaev and Shen \cite{GS} obtained $\delta=1/4-o(1)$ for the function $f(a,b)=a(b+1)$. 
 
The most recent progress in this area was in the case $n=4$, due to a breakthrough by Guth and Katz \cite{GK}. They showed that a set $P$ of points in $\mathbb{R}^2$ determines at least $\Omega\left(|P|^{1-o(1)}\right)$ distinct pairwise distances. When $P=A \times A$ this implies that 
$$f(a,b,c,d)=(a-b)^2+(c-d)^2$$
is a four-variable expander with $\delta = 1-o(1)$.  

The Guth-Katz method uses a novel polynomial partitioning argument that gives strong incidence results for points and lines lying in general position in $\mathbb{R}^3$. It also employs classical results on the flecnode polynomial and the theory of ruled surfaces. These were combined with the Erlangen-type observation of Elekes and Sharir \cite{ES} that the `dual' problem to counting distances was to count rigid motions of the plane, and that this could be parameterised as an incidence problem in $\mathbb{R}^3$.

This methodology was also adapted by Iosevich, Roche-Newton and Rudnev \cite{IRR} to show that 
$$f(a,b,c,d)=ad-bc$$
is likewise a four-variable expander with $\delta = 1-o(1)$. Both this and the Guth-Katz result are sharp up to the $o(1)$ in the exponent, as shown by the case where $A$ is an arithmetic progression.

In this paper we construct a function in only three variables rather than four that is nevertheless an expander with $\delta=1-o(1)$. We also improve on the state of the art for four-variable expanders by giving an example with $\delta=1$ instead of $1-o(1)$. Finally, we give an example of a five-variable expander that achieves $\delta=3-o(1)$.

Unlike the previous results for four-variable  expanders, the results here do not require the Guth-Katz polynomial partitioning technique. Their proofs follow the Elekes-Sharir framework of switching to a `dual' incidence problem, but then rely only on the Szemer\'edi-Trotter incidence theorem for points and lines, and a corresponding result due to Edelsbrunner, Guibas and Sharir for planes.

Moreover, and again unlike the previous results, it is not clear that the bounds for any of these functions should be sharp, even up to logarithmic factors. So we pose the question as to whether or not they can be further improved.

\subsection{Statement of results}

Our expander functions are constructed using \textbf{cross ratios}, which are a key concept in projective geometry. For distinct $a,b,c,d \in \mathbb{R}$, the cross ratio $X(a,b,c,d)$ is defined by 

$$X(a,b,c,d)=\frac{(a-b)(c-d)}{(b-c)(a-d)}.$$ 

Define the following functions: 

\begin{align*}
f(a,b,c)&=X(0,a,b,c) \in \mathbb{R}\\
g(a,b,c,d)&=X(a,b,c,d) \in \mathbb{R}\\
h(a,b,c,d,e)&=\left(X(a,b,c,d),X(a,b,c,e)\right)\in \mathbb{R}^2.
\end{align*}

Our results are then:

\begin{theorem}\label{theorem:result1}
$|f(A)|\gg \frac{|A|^2}{\log |A|}$. 
\end{theorem}

\begin{theorem}\label{theorem:result2}
$|g(A)|\gg |A|^2$. 
\end{theorem}

\begin{theorem}\label{theorem:result3}
$|h(A)|\gg \frac{|A|^4}{\log |A|}$. 
\end{theorem}

It is not obvious that Theorems \ref{theorem:result1}, \ref{theorem:result2} or \ref{theorem:result3} should be sharp. For example, the previous results on four-variable expanders are seen to be sharp up to logarithmic factors by considering the case in which $A$ is an arithmetic progression. However in this case one can verify that 

$$\left|f(A)\right|=\left|g(A)\right|\gg \frac{|A|^3}{\log |A|}$$ 

which is a stronger result than implied by Theorems \ref{theorem:result1} and \ref{theorem:result2}. We therefore ask the following question.

\begin{question}
Can Theorems \ref{theorem:result1}, \ref{theorem:result2} or \ref{theorem:result3} be improved? That is, what are the largest-possible $\delta_1,\delta_2,\delta_3$ for which $|f(A)|\gg |A|^{1+\delta_1}$, $|g(A)|\gg |A|^{1+\delta_2}$ and $|h(A)|\gg |A|^{1+\delta_3}$ for any finite $A \subseteq \mathbb{R}$?
\end{question}

In what follows, Section \ref{section:background} gives necessary background material on projective transformations and cross ratios. Section \ref{section:embedding} identifies projective transformations with points in three-dimensional projective space, and establishes how the transformations' behaviour corresponds to line and plane structures of points. Section \ref{section:incidences} recalls incidence theorems with which to analyse these structures. Finally, Section \ref{section:proofs} uses the material from the previous sections to give proofs of Theorems \ref{theorem:result1}, \ref{theorem:result2} and \ref{theorem:result3}.

\section{Background on projective transformations and cross ratios}\label{section:background}
This section gives the necessary backgound on projective transformations and cross ratios. The material is standard, and can be found in textbooks on projective geometry, for example the book of P. Samuel \cite{PS}.

\subsection{Projective transformations}
We work with the extended real line $\overline{\mathbb{R}}=\mathbb{R} \cup \left\{\infty\right\}$.

\begin{definition}
The group $PSL_2(\mathbb{R})$ of \textbf{projective transformations} of $\overline{\mathbb{R}}$ is defined by $PSL_2(\mathbb{R})=SL_2(\mathbb{R})/\pm I$. It has an action on $\overline{\mathbb{R}}$ given by

$$
\left[
\left(
\begin{array}{cc}
p&q\\
r&s
\end{array}
\right)
\right]
x= \frac{px+q}{rx+s}
$$
which is interpreted in the sense of limits where necessary.
\end{definition}

One can easily check the following facts about $PSL_2(\mathbb{R})$.
\begin{lemma}
The action of $PSL_2(\mathbb{R})$ on $\overline{\mathbb{R}}$ is well-defined. For each $\tau \in PSL_2(\mathbb{R})$, the map $x \mapsto \tau (x)$ is a bijection from $\overline{\mathbb{R}}$ to itself.
\end{lemma}

Another important property of $PSL_2(\mathbb{R})$ is that its action on $\overline{\mathbb{R}}$ is sharply 3-transitive. That is, a projective transformation $\tau \in PSL_2(\mathbb{R})$ is determined exactly by its image at any three elements of $\overline{\mathbb{R}}$, as shown by the following lemma.

\begin{lemma}\label{theorem:transitive}
Let $\mathcal{T}$ be the set of ordered triples of distinct elements of $\overline{\mathbb{R}}$. If  $\left(a,b,c\right)$ and $\left(d,e,f\right)$ are both in $\mathcal{T}$ then there is a unique $\tau \in PSL_2 (\mathbb{R})$ for which $\left(\tau(a),\tau(b),\tau(c)\right)=\left(d,e,f\right)$.
\end{lemma}

\begin{proof}
It suffices to show that for any $(a,b,c)\in \mathcal T$ there is a unique $\tau$ that sends $(a,b,c)$ to $(\infty,0,1)$. Indeed, if this is established then given $(a,b,c)$ and $(d,e,f)$ we can pick $\tau_1,\tau_2$ respectively sending each of them to $(\infty,0,1)$. Then $\mu = \tau_2^{-1} \tau_1$ sends $(a,b,c)$ to $(d,e,f)$, and $\mu$ is unique since any $\mu'$ with this property must satisfy $\tau_2  \mu' =\tau_1$.  

We now show that $$\left[\left(
\begin{array}{cc}
c-a&(c-a)b\\
c-b&(c-b)a
\end{array}
\right)
\right]
$$
is the unique element of $PSL_2(\mathbb{R})$ that sends $(a,b,c)$ to $(\infty,0,1)$. Suppose that 
$$
\tau=
\left[
\left(
\begin{array}{cc}
p&q\\
r&s
\end{array}
\right)
\right]
$$

We have
$\tau(x)=\frac{px+q}{rx+s}$
and so $\tau $ sends $(a,b,c)$ to $(\infty,0,1)$ precisely when 
\begin{enumerate}
\item $ra+s=0$ and $rb+s \neq 0$.
\item $pb+q=0$ and $pa+q \neq 0$.
\item $pc+q=rc+s \neq 0$.
\end{enumerate}
One can check that these three conditions are satisfied precisely when 

$$\left[\left(
\begin{array}{cc}
p&q\\
r&s
\end{array}
\right)\right]
= \left[\left(
\begin{array}{cc}
c-a&(c-a)b\\
c-b&(c-b)a
\end{array}
\right)
\right].
$$
\end{proof}

\subsection{Cross ratios}
Recall that the \textbf{cross ratio} is defined as $X(a,b,c,d)=\frac{(a-b)(c-d)}{(b-c)(a-d)}$ for distinct $a,b,c,d \in \mathbb{R}$. The key importance of the cross ratio is that it is a projective invariant of quadruples, in the following sense.

\begin{lemma}\label{theorem:invariant}
$X(a_1,a_2,a_3,a_4)=X(b_1,b_2,b_3,b_4)$ if and only if there is a projective transformation that sends each $a_i$ to $b_i$.
\end{lemma}

\begin{proof}
Let $\mu$ be the projective transformation that sends $a_i$ to $b_i$ for $i=1,2,3$. We shall show that $X(a_1,a_2,a_3,a_4)=X(b_1,b_2,b_3,b_4)$ if and only if $\mu$ also sends $a_4$ to $b_4$.

First note that $X(a,b,c,d)=\tau_{abc}(d)$ where $\tau_{abc} \in PSL_2(\mathbb{R})$ is the unique projective transformation that sends $\left(a,b,c\right)$ to $\left(\infty,1,0\right)$. To see this it suffices simply to check that 

$$\tau_{abc}=\left[
\left(
\begin{array}{lr}
b-a & (a-b)c\\
c-b & (b-c)a
\end{array}
\right)\right]$$

and then that $\tau_{abc}(d)=X(a,b,c,d)$.

From this observation we know that $X(a_1,a_2,a_3,a_4)=X(b_1,b_2,b_3,b_4)$ if and only if $\tau_{a_1 a_2 a_3}(a_4)=\tau_{b_1b_2b_3}(b_4)$. But $\tau_{a_1a_2a_3}=\tau_{b_1b_2b_3}\circ \mu$ and so by injectivity of $\tau_{b_1 b_2 b_3}$ this occurs precisely when $\mu(a_4)=b_4$. So we are done.
\end{proof}

\section{Points, planes and transformations}\label{section:embedding}

This section contains two results. The first - a `points lemma' - identifies projective transformations from $PSL_2(\mathbb{R})$ with points in $\mathbb{PR}^3$. The second - a `planes lemma' - establishes that the behaviour of transformations corresponds to line and plane structures of their associated points.

\begin{lemma}[Points lemma]\label{theorem:points}
Define $\psi:PSL_2(\mathbb{R})\to \mathbb{PR}^3$ by 
$$\psi\left[\left(
\begin{array}{cc}
p&q\\
r&s	
\end{array}
\right)\right]
=[p,q,r,s].$$
The map $\psi$ is well-defined and injective, and its image is $\mathbb{PR}^3\setminus Q$ where $Q$ is the quadratic surface given by $ps=qr$.
\end{lemma}

\begin{proof}
That $\psi$ is well-defined and injective follows from checking that if $t_1,t_2 \in SL_2(\mathbb{R})$ then $\psi[t_1]=\psi[t_2]$ if and only if $t_1=\pm t_2$. That the image is $\mathbb{PR}^3\setminus Q$ follows from the definition $PSL_2(\mathbb{R})=SL_2(\mathbb{R})/\pm I$. 
\end{proof}

\begin{lemma}[Planes lemma]\label{theorem:planes}
Let $\psi$ be as in the points lemma. For each $(a,b)\in \mathbb{R}^2$ there is a plane $\pi_{ab}\subseteq \mathbb{PR}^3$ such that if $\tau \in PSL_2(\mathbb{R})$ then $\tau(a)=b$ if and only if $\psi(\tau) \in \pi_{ab}$. These planes have the following properties. 
\begin{enumerate}
\item Any triple of distinct planes intersect in a single point. Equivalently, no three planes are colinear.
\item Each $(a,b) \in \mathbb{R}^2$ determines a unique plane (i.e. different pairs of points in $\mathbb{R}^2$ determine different planes).
\item Each pair of distinct planes intersects in a unique line (i.e. different pairs of planes determine different lines).
\item For any $A \subseteq \mathbb{R}$, a point $p \in \mathbb{PR}^3\setminus Q$ is incident to at most $|A|$ of the planes from $\left\{\pi_{ab}:a,b \in A\right\}$. 
\end{enumerate}
\end{lemma}

\begin{proof}
A projective transformation $\tau = \left[\left(
\begin{array}{cc}
p&q\\
r&s
\end{array}
\right)\right]$
sends $a$ to $b$ if and only if $\frac{ap+q}{ar+s}=b$, which is the same as $ap+q-bar-bs=0.$ For fixed $a,b$ this is a linear constraint on $\psi(\tau)=[p,q,r,s]\in \mathbb{PR}^3$ and so describes a plane in $\mathbb{PR}^3$, which we define to be $\pi_{ab}$.

Now that the planes are constructed, we establish properties $1$ to $4$ in turn
\begin{enumerate}

\item Let $(a,b,c)$ and $(d,e,f)$ be two triples of disinct elements of $\mathbb{R}$. There is by Lemma \ref{theorem:transitive} a unique $\tau\in PSL_2(\mathbb{R})$ that sends $(a,b,c)$ to $(d,e,f)$. So $\pi_{ad}\cap \pi_{be} \cap \pi_{cf}=\psi(\tau)$, which is a single point in $\mathbb{PR}^3$.

\item
If $\pi_{ab}=\pi_{cd}$ for some $(a,b)\neq (c,d)$ then we can pick $(e,f)$ such that $\pi_{ab}\cap \pi_{cd} \cap \pi_{ef}$ is either a line or a plane, which contradicts property $1$.

\item
Suppose that $\pi_{ab}\cap \pi_{cd}= \pi_{a'b'}\cap \pi_{c'd'}$. Then $\pi_{ab}\cap \pi_{cd}\cap \pi_{a'b'} = \pi_{a'b'}\cap \pi_{c'd'}.$ But by property $1$ the set on the left hand side is a point, whereas that on the right is a line, unless $\left\{\pi_{ab},\pi_{cd}\right\}=\left\{\pi_{a'b'},\pi_{c'd'}\right\}$.   

\item Let $p$ be a point in $\mathbb{PR}^3\setminus Q$, so that $p = \psi(\tau)$ for some $\tau \in PSL_2(\mathbb{R})$. For each $a \in A$ there is at most one $b \in A$ for which $p$ is incident to $\pi_{ab}$, as otherwise $\tau(a)$ would take two different values. Counting over all $a \in A$ shows that $p$ is incident to at most $|A|$ planes.
\end{enumerate}

\end{proof}

\section{Incidence theorems}\label{section:incidences}

This section records results about incidences, wich we will use to analyse the points, planes and lines from the previous section. We will need two types of incidence theorem: one for incidences between points and lines, and one for incidences between points and planes.

\subsection{Incidences for lines}

For point-line incidences we will use the well-known Szemer\'edi-Trotter \cite{ST} incidence theorem. Given a set $P$ of points and a set $L$ of lines, we write $I(P,L)$ for the number of incidences between points from $P$ and lines from $L$.

\begin{theorem}[Szemer\'edi-Trotter]\label{theorem:ST}
Let $P$ and $L$ be a set of points and lines respectively in $\mathbb{R}^2$. Then $I(P,L)\ll |P|^{2/3}|L|^{2/3}+|P|+|L|$. 
\end{theorem}

This has the following standard corollary.

\begin{corollary}
Let $L$ be a set of lines in $\mathbb{R}^2$. Then the number of points incident to at least $k$ lines in $L$ is $O\left(\frac{|L|^2}{k^3}+\frac{|L|}{k}\right)$
\end{corollary}

\begin{proof}
Let $P$ be the set of points incident to at least $k$ lines in $L$. Then $|P|k\leq I(P,L)$. Comparing to the upper bound in the Szemer\'edi-Trotter theorem shows that $|P|\ll \frac{|L|^2}{k^3}+\frac{|L|}{k}$ as required. 
\end{proof}

\subsection{Incidences for planes}

For point-plane incidences we will use the following result of Edelsbrunner, Guibas and Sharir\footnote{In the original paper \cite{EGS} this bound is multiplied by a factor of the form $|P|^{\epsilon}|\Pi|^{\epsilon}$. However e.g. Apfelbaum and Sharir \cite{AS} report that this additional factor can be eliminated with more careful analysis and so we use the refined version here.} \cite{EGS}.  

\begin{theorem}[Edelsbrunner-Guibas-Sharir]\label{theorem:EGS}
Let $P$ and $\Pi$ be a set of points and planes respectively in $\mathbb{R}^3$. If no three planes are colinear then $I(P,\Pi)\ll |P|^{4/5}|\Pi|^{3/5}+|P|+|\Pi|$. 
\end{theorem}

As with the Szemer\'edi-Trotter theorem, there is a standard corollary.

\begin{corollary}
Let $\Pi$ be a set of planes in $\mathbb{R}^3$, no three of which are colinear. Then the number of points incident to at least $k$ planes in $\Pi$ is $O\left(\frac{|\Pi|^3}{k^5}+\frac{|\Pi|}{k}\right).$
\end{corollary}

\begin{proof}
Let $P$ be the set of points incident to at least $k$ planes in $\Pi$. Then $|P|k\leq I(P,\Pi)$. Comparing to the upper bound in the Edelsbrunner-Guibas-Sharir theorem shows that $|P|\ll \frac{|\Pi|^3}{k^5}+\frac{|\Pi|}{k}$ as required. 
\end{proof}

\section{Proving Theorems \ref{theorem:result1}, \ref{theorem:result2} and \ref{theorem:result3}\label{section:proof}}
\label{section:proofs}

This section use the results so far established to prove Theorems \ref{theorem:result1}, \ref{theorem:result2} and \ref{theorem:result3}.

\subsection{Proof of Theorem \ref{theorem:result1}}
Write $E_1(A)$ for the number of solutions to the equation

\begin{equation}\label{eq:energy1}
X(0,a_1,a_2,a_3)=X(0,b_1,b_2,b_3)
\end{equation}

with each of the $a_i$ and $b_i$ in $A$. Write $\mu_1(x)$ for the number of $a_1,a_2,a_3 \in A$ with $X(0,a_1,a_2,a_3)=x$. Then $\sum_{x \in f(A)}\mu_1(x)\approx|A|^3$, and Cauchy-Schwarz implies that

$$|A|^6 \approx \left(\sum_{x \in f(A)}\mu_1(x)\right)^2 \leq |f(A)|E_1(A).$$

So it suffices to show $E_1(A)\ll |A|^4 \log |A|$. By Lemma \ref{theorem:invariant}, equation (\ref{eq:energy1}) is satisfied precisely when there exists $\tau \in PSL_2(\mathbb{R})$ that fixes $0$ and sends each $a_i$ to $b_i$. Define 

$$T_1=\bigcup_{a,b \in A}\left\{\tau:\tau(0)=0,\tau(a)=b\right\}$$

and write $N_1(\tau)$ for the number of $(a,b)\in A^2$ for which $\tau(a)=b$. Then

\begin{equation*}
\label{eq:energy1a}
E_1(A) \ll \sum_{\tau \in T_1}N_1(\tau)^3.
\end{equation*}

Let $\psi$ be as in the points lemma. Define a set of points by $P=\psi(T_1)$, and a set of lines by $L=\left\{\pi_{ab}\cap \pi_{00}:a,b \in A\right\}$ so that $|L|\approx |A|^2$. The points and lines all lie in the plane $\pi_{00}$. Moreover, if we write $M_1(p)$ for the number of lines from $L$ incident to a point $p$, then $N_1(\tau)=M_1(\psi(\tau))$. So we have

\begin{equation*}
\label{eq:energy1b}
E_1(A) \ll \sum_{p \in P}M_1(p)^3.
\end{equation*}

For each $j \in \mathbb{N}$ write $P_j$ for the set of $p \in P$ with $M_1(p) \in [2^j,2^{j-1})$. Then from the corollary to Szemer\'edi-Trotter we have

\begin{align*}
E_1(A)&\ll \sum_{j=0}^{\log |A|}|P_j|2^{3j}\ll \sum_{j=0}^{\log |A|} \left(\frac{|L|^2}{2^{3j}}+\frac{|L|}{2^j}\right)2^{3j}\approx |A|^4 \log |A|
\end{align*}

as required. \qed

\subsection{Proof of Theorem \ref{theorem:result2}}

Write $E_2(A)$ for the number of solutions to the equation

\begin{equation}\label{eq:energy2}
X(a_1,a_2,a_3,a_4)=X(b_1,b_2,b_3,b_4)
\end{equation}

Using Cauchy-Schwarz as in Theorem \ref{theorem:result1} shows that $|g(A)|\gg \frac{|A|^8}{E_2(A)}$, so it suffices to show that $E_2(A)\ll |A|^6$. Equation (\ref{eq:energy2}) is satisfied precisely when there exists $\tau \in PSL_2(\mathbb{R})$ that sends $a_i$ to $b_i$ for each $i$. Define

$$T_2=\bigcup_{a,b \in A}\left\{\tau:\tau(a)=b\right\}$$

and write $N_2(\tau)$ for the number of $(a,b)\in A^2$ for which $\tau(a)=b$. Then

\begin{equation*}
\label{eq:energy2a}
E_2(A) \ll \sum_{\tau \in T_2}N_2(\tau)^4.
\end{equation*}

Let $\psi$ be as in the points lemma. Define a set of points by $P=\psi(T_2)$, and a set of planes by $\Pi=\left\{\pi_{ab}:a,b \in A\right\}$ so that $|\Pi|\approx |A|^2$. If we write $M_2(p)$ for the number of planes from $\Pi$ incident to a point $p$, then $N_2(\tau)=M_2(\psi(\tau))$. So we have

\begin{equation*}
\label{eq:energy2b}
E_2(A) \ll \sum_{p \in P}M_2(p)^4.
\end{equation*}

For each $j \in \mathbb{N}$ write $P_j$ for the set of $p \in P$ with $M_2(p) \in [2^j,2^{j+1})$. Then from the corollary to Edelsbrunner-Guibas-Sharir we have

\begin{align*}
E_2(A)&\ll \sum_{j=0}^{\log |A|}|P_j|2^{4j}\ll \sum_{j=0}^{\log |A|} \left(\frac{|\Pi|^3}{2^{5j}}+\frac{|\Pi|}{2^j}\right)2^{4j}\ll
|A|^6 \sum_{j=0}^{\infty} \frac{1}{2^{j}} \approx |A|^6. 
\end{align*}

as required. \qed

\subsection{Proof of Theorem \ref{theorem:result3}}
Write $E_3(A)$ for the number of solutions to

\begin{equation}\label{eq:energy3}
\left(X(a_1,a_2,a_3,a_4),X(a_1,a_2,a_3,a_5)\right)=\left(X(b_1,b_2,b_3,b_4),X(b_1,b_2,b_3,b_5)\right)
\end{equation}

with each of the five $a_i$ and five $b_i$ in $A$. By Cauchy-Schwarz, $|h(A)|\gg \frac{|A|^{10}}{E_3(A)}$ so it suffices to show that $E_3(A)\ll |A|^6 \log |A|$. Now equation (\ref{eq:energy3}) is satisfied precisely when there is a projective transformation $\tau$ that sends each of the five $a_i$ to $b_i$. So following the proof of Theorem \ref{theorem:result2} we can find $P$ and $\Pi$ with $|\Pi|\approx |A|^2$ for which 

$$E_3(A)\ll \sum_{j=1}^{\log |A|}\left(\frac{|\Pi|^2}{2^{5j}}+\frac{|\Pi|}{2^j}\right)2^{5j}\approx |A|^6 \log |A|$$ as required. \qed

\section*{Acknowledgements}
The author is grateful to Misha Rudnev for useful conversations, in particular for highlighting the existence of the incidence theorem for planes.

\bibliographystyle{plain}
\bibliography{crossratiobib}

\end{document}